\documentclass[11pt]{amsart}
\usepackage{amscd,amsthm,amssymb,amsfonts,amsmath,euscript}

\theoremstyle{plain}
\newtheorem{thm}{Theorem}[section]
\newtheorem{lemma}[thm]{Lemma}
\newtheorem{prop}[thm]{Proposition}
\newtheorem{cor}[thm]{Corollary}

\theoremstyle{definition}
\newtheorem{defn}[thm]{Definition}

\theoremstyle{remark}
\newtheorem{remark}[thm]{Remark}
\newtheorem*{thank}{{\bf Acknowledgments}}
\newcommand{\nc}{\newcommand}

\def\makeop#1{\expandafter\def\csname#1\endcsname
  {\mathop{\rm #1}\nolimits}\ignorespaces}
\makeop{Hom}   \makeop{End}   \makeop{Aut}   \makeop{Isom}  \makeop{Pic} 
\makeop{Gal}   \makeop{ord}   \makeop{Char}  \makeop{Div}   \makeop{Lie} 
\makeop{PGL}   \makeop{Corr}  \makeop{PSL}   \makeop{sgn}   \makeop{Spf}
\makeop{Spec}  \makeop{Tr}    \makeop{Nr}    \makeop{Fr}    \makeop{disc}
\makeop{Proj}  \makeop{supp}  \makeop{ker}   \makeop{im}    \makeop{dom}
\makeop{coker} \makeop{Stab}  \makeop{SO}    \makeop{SL}    \makeop{SL}
\makeop{Cl}    \makeop{cond}  \makeop{Br}    \makeop{inv}   \makeop{rank}
\makeop{id}    \makeop{Fil}   \makeop{Frac}  \makeop{GL}    \makeop{SU}
\makeop{Trd}   \makeop{Sp}    \makeop{Tr}    \makeop{Trd}   \makeop{diag}
\makeop{Res}   \makeop{ind}   \makeop{depth} \makeop{Tr}    \makeop{st}
\makeop{Ad}    \makeop{Int}   \makeop{tr}    \makeop{Sym}   \makeop{can}
\makeop{length}\makeop{SO}    \makeop{torsion} \makeop{GSp} \makeop{Ker}
\makeop{Adm}
\def\makebb#1{\expandafter\def
  \csname bb#1\endcsname{{\mathbb{#1}}}\ignorespaces}
\def\makebf#1{\expandafter\def\csname bf#1\endcsname{{\bf
      #1}}\ignorespaces} 
\def\makegr#1{\expandafter\def
  \csname gr#1\endcsname{{\mathfrak{#1}}}\ignorespaces}
\def\makescr#1{\expandafter\def
  \csname scr#1\endcsname{{\EuScript{#1}}}\ignorespaces}
\def\makecal#1{\expandafter\def\csname cal#1\endcsname{{\mathcal
      #1}}\ignorespaces} 

\def\doLetters#1{#1A #1B #1C #1D #1E #1F #1G #1H #1I #1J #1K #1L #1M
                 #1N #1O #1P #1Q #1R #1S #1T #1U #1V #1W #1X #1Y #1Z}
\def\doletters#1{#1a #1b #1c #1d #1e #1f #1g #1h #1i #1j #1k #1l #1m
                 #1n #1o #1p #1q #1r #1s #1t #1u #1v #1w #1x #1y #1z}
\doLetters\makebb   \doLetters\makecal  \doLetters\makebf
\doLetters\makescr 
\doletters\makebf   \doLetters\makegr   \doletters\makegr
     \def\qed{\qedmark\medbreak}%
\def\qedmark{{\enspace\vrule height 6pt width 5pt depth 1.5pt}}%
    \def\setminus{\smallsetminus}

\normalsize

\makeop{Bl}

\def\Fpbar{\overline{\bbF}_p}
\def\Fp{{\bbF}_p}

\def\Qp{{\bbQ}_p}

\def\Zp{{\bbZ}_p}
\def\Qbar{\overline{\bbQ}}

\newcommand{\Z}{\mathbb Z}
\newcommand{\Q}{\mathbb Q}

\newcommand{\C}{\mathbb C}
\newcommand{\F}{\mathbb F}

\newcommand{\npr}{\noindent }

\newcommand{\<}{\langle}   
\renewcommand{\>}{\rangle} 

\nc{\embed}{\hookrightarrow}

\newcommand{\ch}{characteristic }
\newcommand{\ac}{algebraically closed }
\newcommand{\dieu}{Dieudonn\'{e} }

\nc{\ol}{\overline}
\nc{\wt}{\widetilde}
\nc{\opp}{\mathrm{opp}}

\def\ci{{\rm ci}}

\def\vol{{\rm vol}}
\def\char{{\rm char}}
\makeop{Ram}
\makeop{Rep}
\newcommand{\DM}{\mathcal{DM}}

\begin{document}
\renewcommand{\thefootnote}{\fnsymbol{footnote}}
\setcounter{footnote}{-1}
\numberwithin{equation}{section}


\title[Finiteness of endomorphism rings]
{On finiteness of endomorphism rings of abelian varieties}
\author{Chia-Fu Yu}
\address{
Institute of Mathematics, Academia Sinica and NCTS (Taipei Office)\\
6th Floor, Astronomy Mathematics Building \\
No. 1, Roosevelt Rd. Sec. 4 \\ 
Taipei, Taiwan }
\email{chiafu@math.sinica.edu.tw}


\begin{abstract}
  The endomorphism ring $\End(A)$ of an abelian variety $A$ is an order in a
  semi-simple algebra over $\Q$. The co-index of $\End(A)$ is the index
  to a maximal order containing it. We show that for abelian varieties
  of fixed dimension over any field 
  of \ch $p>0$, the $p$-exponents of the co-indices of their
  endomorphism rings are bounded. We also give a few applications to
  this finiteness result.   
\end{abstract} 

\maketitle


\def\ci{{\rm ci}} 

\section{Introduction}
\label{sec:01}
Endomorphism algebras of abelian varieties are important
objects for studying abelian varieties. For example, a
theorem of Grothendieck tells us that any
isogeny class of abelian varieties over a field of \ch $p>0$ that has
sufficient many complex multiplications is defined over a finite
field. See Oort \cite{oort:cm} and \cite{yu:cm} for more details.  
Endomorphism algebras have been studied extensively in the
literature; see Oort \cite{oort:endo} for many detailed and
interesting discussions and quite complete references therein. 
Thanks to Tate \cite{tate:eav},
Zarhin \cite{zarhin:end},  
Faltings \cite{faltings:end}, and de Jong \cite{dejong:homo}, 
we have now a fundamental approach using Tate modules (and its
analogue at $p$) to study these endomorphism algebras. 
However, not much is known for their endomorphism rings 
except for the one-dimensional case (see
Theorem~\ref{deuring}). 
In \cite{waterhouse:thesis} Waterhouse determined all possible
endomorphism rings for ordinary elementary abelian varieties over
a finite field (see \cite{waterhouse:thesis}, Theorem 7.4 for more
details).  
 

Let $A_0$ be an abelian variety over a field $k$. Denote by
$[A_0]_k$ the isogeny class of $A_0$ over $k$. It is well-known that the
endomorphism ring $\End(A_0)$ is an order of the semi-simple $\Q$-algebra
$\End^0(A_0):=\End(A_0)\otimes \Q$. A general question is what we can say
about the endomorphism rings $\End(A)$ of abelian varieties $A$ in the
isogeny class $[A_0]_k$. In the paper we consider the basic question:
how many isomorphism classes of the endomorphism rings $\End(A)$ of
abelian varieties $A$ in a fixed isogeny class $[A_0]_k$?

We define a natural numerical invariant for orders in a semi-simple
algebra which measures how far it is from a maximal order. 
Let $B$ be a finite-dimensional semi-simple algebra over $\Q$, and $O$
an order of 
$B$. Define the {\it co-index of $O$}, which we denote ${\rm ci}(O)$, 
to be the index $[R:O]$, where $R$ is a maximal order of $B$ 
containing $O$. The invariant $\ci(O)$ is independent of the choice of
$R$ (see Lemma~\ref{21}). 
For any prime $\ell$, let $v_\ell$ be the discrete valuation on 
$\Q$ at the prime $\ell$ normalized so that $v_\ell(\ell)=1$. The main
results of this paper are: 

\begin{thm}\label{11} 
  Let $g\ge 1$ be an integer. There is an positive
  integer $N$ only depending on $g$ such that $v_p({\rm
  ci}(\End(A)))<N$ for any $g$-dimensional abelian variety over any
  field of \ch $p>0$.
\end{thm}


\begin{thm}\label{12}
  Let $g\ge 1$ be an integer. There are only finitely
  many isomorphism classes of rings $\End(A)\otimes \Z_p$ 
  for all $g$-dimensional abelian varieties $A$ over any field of \ch
  $p>0$.  
\end{thm}

One can also deduce easily from Theorem~\ref{11} the following:

\begin{cor}\label{13}
   Let $g\ge 1$ be an integer. There are only finitely
  many isomorphism classes of endomorphism rings of
  $g$-dimensional supersingular abelian varieties over an \ac
  field $k$ of \ch $p>0$.
\end{cor}

As pointed out by the referee, Theorem~\ref{11} generalizes the
following classical result of Deuring \cite{deuring}. See Lang's book
\cite{lang:ef}, Chapter 13 for a modern exposition. 

\begin{thm}[\bf Deuring]\label{deuring}  
  Let $E$ be an elliptic curve over an \ac field of prime \ch
  $p$. Then its endomorphism ring $\End(E)$ is either $\Z$, a maximal
  order in the definite quaternion $\Q$-algebra of discriminant $p$, 
  or an order in an imaginary
  quadratic field whose conductor is prime to $p$. In particular, the
  index of $\End(E)$ in a maximal order of $\End^0(E)$ is prime to $p$. 
\end{thm}

Note that by a theorem of Li-Oort \cite{li-oort}, the supersingular
locus $\calS_g$ of the Siegel moduli space $\calA_{g}\otimes \Fpbar$ has
dimension $[{g^2}/{4}]$. In particular, there are infinitely
many non-isomorphic supersingular abelian varieties. It is
a priori not obvious why there should be only finitely many
isomorphism classes in their endomorphism rings. However, since all of
them are given by an isogeny of degree $p^{g(g-1)/2}$ from a
superspecial one (see Li-Oort \cite{li-oort}), 
the finiteness result might be expected. This is
indeed the idea of proving Theorem~\ref{11}. 

The proof of Theorem~\ref{11} uses the following  
universal bounded property due to Manin \cite{manin:thesis}: 
for a fixed integer $h\ge 1$, the degrees
of the {\it minimal isogenies}
$\varphi: X_0\to X$, for all $p$-divisible 
groups $X$ of height $h$ over an \ac field of fixed \ch $p$, are
bounded.  See Section~\ref{sec:04} for the definition and properties
of minimal isogenies. 
F.~Oort asks (in a private conversation) the following question: 
if $X$ is equipped with an
action $\iota$ by an order $\calO$ of a finite-dimensional 
semi-simple algebra over $\Q_p$, 
is there an action $\iota_0$ of $\calO$ on $X_0$ so that the 
minimal isogeny $\varphi$ becomes 
$\calO$-linear? Clearly, such a map $\iota_0:\calO\to \End(X_0)$ is
unique if it exists. The motivation of this question is looking for
 a good notion of minimal isogenies when one considers abelian
 varieties with additional 
structures (polarizations and endomorphisms). 
We confirm his question with positive answer in Section~\ref{sec:04} 
(see Proposition~\ref{47}). This also plays a role in the proof of
Theorem~\ref{11}.  


Theorem~\ref{11} is sharp at least when the ground field $k$ is
algebraically closed. 
Namely, for any prime $\ell\neq {\rm char} (k)$,   
the finiteness for $v_\ell({\rm ci}(\End(A)))$ does not hold in
general. Indeed, we show (see Section \ref{sec:05})

\begin{prop}\label{14}
Let $p$ be a prime number or zero. There exists an abelian variety
$A_0$ over an \ac field $k$ of \ch $p$ so that for any prime $\ell\neq
p$ and any integer $n\ge 1$, there exists an $A \in [A_0]_k$ such that
$v_\ell({\rm ci}(\End(A)))\ge n$. 
\end{prop}

In fact, elliptic curves already provide such examples in
Proposition~\ref{14}. For these examples, there are infinitely
many isomorphism classes of rings $\End(A)\otimes \Z_\ell$ in the
isogeny class for each prime $\ell$ prime to the \ch of ground
field.

The finiteness result (Corollary~\ref{13}) gives rise to 
a new refinement on the supersingular
locus $\calS_g$ arising from arithmetic. We describe now this
``arithmetic'' refinement in the special case where $g=2$. 
Let $V$ be an irreducible component of the supersingular locus
$\calS_2$ in the Siegel 3-fold (with an auxiliary prime-to-$p$ level
structure) over $\Fpbar$. 
It is known 
that $V$ is isomorphic to
$\bfP^1$ over $\Fpbar$. We fix an isomorphism and choose an
appropriate $\F_{p^2}$-structure on $V$ (see Subsection~\ref{sec:53} for
details). For any point $x$ in the Siegel moduli space, write $c_p(x)$
for $v_p(\ci(\End(A_x)))$, where $A_x$ is the underlying abelian
variety of the object $(A_x,\lambda_x,\eta_x)$ corresponding to 
the point $x$. For each integer $m\ge 0$, let
\[ V_m:=\{x\in V; c_p(x)\le m\}. \]
The collection $\{V_m\}_{m\ge 0}$ forms an increasing sequence of 
closed subsets of $V=\bfP^1$. We have (see
Subsection~\ref{sec:54})
\[ V_0=\dots=V_3\subset V_4=V_5\subset V_6=V, \]
and 
\[ V_0=\bfP^1(\F_{p^2}), \quad  V_4=\bfP^1(\F_{p^4}). \]
This refines the standard consideration on $\calS_2$ by superspecial
and non-superspecial points. 



This paper is organized as follows. Sections 2-4 are devoted to the proof
of Theorems~\ref{11} and \ref{12}. Section
~\ref{sec:02} reduces to an analogous statement for $p$-divisible groups
over an \ac field. Section \ref{sec:03} provides necessary information
about minimal \dieu modules. In Section \ref{sec:04} we use 
minimal isogenies to conclude the finiteness of co-indices of
endomorphism rings of $p$-divisible groups, and finish the proof of
Theorem~\ref{12}. 
Section~\ref{sec:05} provides
examples which particularly show that the $\ell$-co-index of the 
endomorphism rings can be arbitrarily large for any prime $\ell\neq p$. 
A special case for the ``arithmetic'' refinement is treated there.


\begin{thank}
  Obviously the present work relies on the work of Manin
  \cite{manin:thesis} and uses the notion of minimal isogenies whose
  significance is pointed out in Li-Oort \cite{li-oort}. The author
  wishes to thank them for the influential papers. 
  He thanks  C.-L.~Chai, U. G\"ortz, F.~Oort and J.-D.~Yu for helpful
  discussions and comments, and the
  referee for careful reading and helpful comments that improve the
  exposition significantly. 
  The manuscript is prepared during the author's stay at 
  Universit\"at Bonn, and some revision is made 
  at l'Institut des Hautes \'Etudes
  Scientifiques. He acknowledges the institutions for 
  kind hospitality and excellent working
  conditions. 
  The research is partially supported by grants NSC
  97-2115-M-001-015-MY3 and AS-98-CDA-M01.
\end{thank}
\section{Reduction steps of Theorem~\ref{11}}
\label{sec:02}

\subsection{Co-Index}
\label{sec:21} 
Let $K$ be a number field, and $O_K$ the ring of integers. 
Denote by $K_v$ the completion of $K$ at a place $v$ of $K$, and
$O_{K_v}$ the ring of integers when $v$ is finite.
Let $B$ be a finite-dimensional semi-simple algebra over $K$,
and let $O$ be an $O_K$-order of $B$. The {\it co-index} of $O$, 
written as $\ci(O)$, is defined to be the index $[R:O]$, where $R$ is a
maximal order of $B$ containing $O$. 
We define the co-index similarly for an order of a finite-dimensional 
semi-simple algebra over a $p$-adic local field. 
For each finite place $v$ of $K$, we write $R_v:=R\otimes_{O_K}
O_{K_v}$ and $O_v:=O\otimes_{O_K} O_{K_v}$. 
From the integral theory of semi-simple algebras
(see Reiner \cite{reiner:mo}), each $R_v$ is a maximal order of
$B\otimes_K K_v$ and we
have $R/O\simeq \oplus_v R_v/O_v$, where $v$ runs through all finite
places of $K$. It follows that 
\begin{equation}
  \label{eq:20}
  \ci(O)=\prod_{v:\text{finite}} \ci (O_v), \quad
  \ci(O_v):=[R_v:O_v]. 
\end{equation}
As the algebra $B$ is determined by
$O$, the co-index $\ci(O)$ makes sense without mentioning the
algebra $B$ containing it.   

\begin{lemma}\label{21}
  The co-index $\ci (O)$ is independent of the choice of a maximal
  order containing it. 
\end{lemma}
\begin{proof}
  Using the product formula (\ref{eq:20}), it suffices to show 
  the local version of the statement. Therefore, we may assume that 
  $K$ is a $p$-adic local field. If $R'$ is another maximal order
  containing $O$, then $R'=g R g^{-1}$ for some element $g\in
  B^\times$. Since in this case $[R:O]=\vol(R)/\vol(O)$ for any
  Haar measure on $B$, the statement then follows from the equality 
  $\vol(R)=\vol(g R g^{-1})$.  \qed
\end{proof}

\subsection{Base change}
\label{sec:22}

\begin{lemma}\label{22}
  Let $A$ be an abelian variety over a field $k$ and let $k'$ be a
  field extension of $k$, then the inclusion 
  $\End_k(A)\to \End_{k'}(A\otimes
  k')$ is co-torsion-free, that is, the quotient is torsion
  free. Furthermore, we have 
\begin{equation}\label{eq:21}
 \ci[\End_k(A)]\,|\,\ci[\End_{k'}(A\otimes k')]. 
\end{equation}
\end{lemma}

\begin{proof}
  The first statement follows from Oort \cite{oort:endo}, Lemma 2.1. 
  For the second statement, we
  choose a maximal order $O_1$ of $\End^0_k(A)$ containing
  $\End_k(A)$. Let $O_2$ be a maximal order of $\End^0_{k'}(A\otimes
  k')$ containing $O_1$ and $\End_{k'}(A\otimes k')$. Since
  $\End_k(A)=\End^0_k(A)\cap \End_{k'}(A\otimes k')$, we have
  the inclusion $O_1/\End_k(A)\subset O_2/\End_{k'}(A\otimes
  k')$. This proves the lemma. \qed
\end{proof}




By Lemma~\ref{22}, we can reduce Theorem~\ref{11} to the case where $k$ is 
algebraically closed.

\subsection{Reduction to $p$-divisible groups}
\label{sec:23}

\begin{lemma}\label{23}
  Let $A$ be an abelian variety over a field $k$. Let $\ell$ be a
  prime, possibly equal to $\char (k)$. The inclusion map
  $\End_k(A)\otimes \Z_\ell \to \End_k(A[\ell^\infty])$ is
  co-torsion-free. Here $A[\ell^\infty]$ denotes the associated
  $\ell$-divisible group of $A$. 
%
\end{lemma}
\begin{proof}
  When $\ell\neq \char (k)$, this is elementary and well-known;
  see Tate \cite{tate:eav}, p.~135. The same argument also shows the case 
  when $\ell=\char (k)$.\qed   
\end{proof}

We remark that 
for an arbitrary ground field $k$, the endomorphism algebra
$\End^0_k(A[\ell^\infty]):=\End_k(A[\ell^\infty])\otimes_{\Z_\ell}
\Q_\ell$ of the associated $\ell$-divisible group 
$A[\ell^\infty]$ of an abelian variety $A$ over $k$,  
where $\ell$ is a prime $\neq \char (k)$, may not be
semi-simple; see Subsection~\ref{sec:55}. Therefore, the numerical
invariant $\ci(\End_k(A[\ell^\infty]))$ may not be defined in general. 
Analogously, in the case where $\char (k)=p>0$, 
the endomorphism algebra $\End^0_k(X):=\End_k(X)\otimes_{\Z_p} \Q_p$ of a 
$p$-divisible group $X$ over $k$ may not be semi-simple, and hence
the numerical invariant $\ci(\End_k(X))$ may not be defined in general,
either. See Subsection~\ref{sec:56}. However, when the ground field
$k$ is algebraically closed, both $\ci(\End(A[\ell^\infty]))$ and
$\ci(\End(A[p^\infty]))$ are always defined for any abelian variety $A$. 

\begin{lemma}\label{24} 
Let $A$ be an abelian variety over an \ac field $k$ of \ch 
$p>0$. Then one has 
\begin{equation}\label{eq:22}
 v_p(\ci(\End_k(A))) \le v_p(\ci(\End_k(A[p^\infty]))).   
\end{equation}
\end{lemma}
\begin{proof}
Let $R$ be a maximal order of $\End(A)\otimes \Q_p$ containing
$\End(A)\otimes \Zp$. 
Then there is an isogeny $\varphi: A\to A'$ of $p$-power 
degree over $k$ such that $\End(A')\otimes \Zp=R$. We may assume that 
the degree of this isogeny is minimal among isogenies with this
property. Then we have $\End(A[p^\infty])\subset \End(A'[p^\infty])$. As 
$\End(A)\otimes \Qp\cap \End(A[p^\infty])=\End(A)\otimes \Zp$, we have 
the inclusion $R/(\End(A)\otimes \Zp)\subset 
\End(A'[p^\infty])/(\End(A[p^\infty]))$. This yields the inequality 
(\ref{eq:22}). \qed  
\end{proof}

By Lemmas~\ref{22}
and~\ref{24}, Theorem~\ref{11} follows from the following theorem.


\begin{thm}\label{25}
  Let $k$ be an \ac field of \ch $p>0$ and let $h\ge 1$ be a fixed integer.
  Then there is an integer $N>1$, depending only on $h$, such that for
  any $p$-divisible group $X$ of height $h$ over $k$, one has
  $v_p(\ci(\End(X)))\le N$.     
\end{thm}

\section{Minimal \dieu modules}
\label{sec:03}

\subsection{Notation}
\label{sec:31}
In Sections~\ref{sec:03} and~\ref{sec:04}, we let
$k$ denote an \ac field of \ch $p>0$. 
Let $W:=W(k)$ be the ring of Witt
vectors over $k$, and $B(k)$ be the fraction field of $W(k)$. Let
$\sigma$ be the Frobenius map on $W$ and $B(k)$, respectively. 
For each $W$-module $M$ and each subset $S\subset M$, we denote by $\<S\>_W$
the $W$-submodule generated by $S$. Similarly, $\<S\>_{B(k)}\subset
M\otimes \Q_p$ denotes the vector subspace over $B(k)$ generated by
$S$. In this paper we use the covariant \dieu theory.
\dieu modules considered here are assumed to be finite and free as 
$W$-modules. Let $\calD\calM$ denote the category of \dieu modules
over $k$.

To each rational number $0\le \lambda\le 1$, one associates coprime
non-negative integers $a$ and $b$ so that $\lambda=b/(a+b)$.    
For each pair $(a,b)\neq (0,0)$ of coprime non-negative integers, write
$M_{(a,b)}$ for the \dieu module $W[F,V]/(F^a-V^b)$.

We write a Newton polygon or a slope sequence $\beta$ as a finite
formal sum:
\[ \sum_i r_i \lambda_i\quad \text{or}\quad \sum_i r_i (a_i,b_i), \]
where
each $0\le \lambda_i\le 1$ is a rational number, $r_i\in \bbN$ is a
positive integer, 
and $(a_i,b_i)$ is the pair associated to $\lambda_i$ (By convention,
the multiplicity of $\lambda_i$ is $b_i r_i$).
The Manin-\dieu Theorem  (\cite{manin:thesis}, Chap. II,
``Classification Theorem'', p.~35) asserts 
that for any \dieu module $M$ over $k$, there are distinct coprime
non-negative pairs $(a_i,b_i)\neq (0,0)$, and positive integers $r_i$, for
$i=1,\dots, s$, such that there is an isomorphism of $F$-isocrystals
\begin{equation}
  \label{eq:31}
  M\otimes \Q_p\simeq \bigoplus_{i=1}^s (M_{(a_i,b_i)}\otimes 
\Q_p)^{\oplus r_i}.
\end{equation}
Moreover, the pairs $(a_i,b_i)$ and integers $r_i$ are uniquely
determined by 
$M$. The Newton polygon of $M$ is defined to be 
$\sum_{i=1}^s r_i (a_i,b_i)$; the rational numbers
$\lambda_i=b_i/(a_i+b_i)$ are 
called the slopes of $M$. The Newton polygon of he \dieu module
$M_{(a,b)}$ 
above has single slope $\lambda=b/(a+b)$.  


  
The $F$-subisocrystal $N_{\lambda_i}$ of $M\otimes \Qp$ that
corresponds to the factor $(M_{(a_i,b_i)}\otimes 
\Q_p)^{\oplus r_i}$ in (\ref{eq:31}) is unique and is called the {\it
  isotypic component of $M\otimes \Qp$ of slope
  $\lambda_i=b_i/(a_i+b_i)$}. A \dieu module or an $F$-isocrystal is
called {\it isoclinic} if it has single slope, or equivalently
$M\otimes \Q_p$ is an isotypic component of itself. 

If $M$ is a \dieu module over $k$, the endomorphism ring 
$\End(M)=\End_\DM(M)$
is the ring of endomorphisms on $M$ in the category $\DM$; we write
$\End^0(M):=\End(M)\otimes_{\Zp} \Q_p$ for the endomorphism algebra of
$M$. 
If the Newton polygon of $M$ is $\sum_{i=1}^s r_i
(a_i,b_i)$, then the endomorphism algebra of $M$ is isomorphic to 
the product of the matrix algebras $M_{r_i}(\End^0(M_{(a_i,b_i)}))$. 

\begin{lemma}\label{31} \
\begin{itemize}
  \item[(1)] The endomorphism algebra $\End^0(M_{(a,b)})$ is
    isomorphic to 
    \begin{equation}
      \label{eq:32}
      B(\F_{p^n})[\Pi'], \quad (\Pi')^n=p^b, \quad c \Pi'= \Pi'
      \sigma(c), \quad \forall\, c\in B(\F_{p^n}).
    \end{equation}
   Therefore, $\End^0(M_{(a,b)})$ is a central division algebra over
   $\Qp$ of degree $n^2$ with Brauer invariant $b/n$. 
 \item [(2)] The maximal order of the division algebra
   $B(\F_{p^n})[\Pi']$ is $W(\F_{p^n})[\Pi]$, where $\Pi=(\Pi')^m p^{m'}$
   for some integers $m$ and $m'$ such that $bm+nm'=1$, subject to the
   following relations 
   \begin{equation}
     \label{eq:33}
     \Pi^n=p, \quad\text{and}\quad c
   \Pi=\Pi \sigma^{m}(c)\quad  \forall\, c\in W(\F_{p^n}) 
   \end{equation}
\end{itemize}
\end{lemma}
\begin{proof}
  This is certainly well-known; we provide a proof for the reader's
  convenience. Note that using $(\Pi')^n=p^b$ one sees that the
  element $\Pi$ in (2) independent of the choice of the integers $m$
  and $m'$.

  (1) The $F$-isocrystal $N:=M_{(a,b)}\otimes \Qp$ is generated by the
      element $e_0:=1$. Put $e_i:=F^i e_0$ for $i=1,\dots, n-1$; the
      vectors $e_0,\dots,e_{n-1}$ form a $B(k)$-basis for $N$. Since $N$ is
      generated by $e_0$ (as an $F$-isocrystal) and $(F^n-p^b)e_0=0$,
      any endomorphism $\varphi\in \End(N)$ is determined by the
      vector $\varphi(e_0)$ and this vector lies in the subspace
      $\<e_0,\dots, e_{n-1}\>_{B(\F_{p^n})}$. Let $\Pi'$ be the element in
      $\End(N)$ such that $\Pi'(e_0)=e_1$, and for each element $c\in
      B(\F_{p^n})$, let $\varphi_c$ be the endomorphism such that
      $\varphi_c(e_0)=c e_0$. It is not hard to see that 
      the endomorphism algebra $\End(N)$
      is generated by elements $\Pi'$ and $\varphi_c$ for all $c\in
      B(\F_{p^n})$.   
      One checks that $\varphi_c \Pi'= \Pi'
      \varphi_{\sigma(c)}$ for all $c\in B(\F_{p^n})$. This proves the
      first part of (1).  
      One extends the valuation $v_p$ on $\Qp$ naturally to the
      division algebra 
      $B(\F_{p^n})[\Pi']$. According to the definition (we use the
      normalization in \cite{pierce}, see p.~338), the Brauer
      invariant is given by 
      $v_p(\Pi')$, which is equal to $b/n$. Therefore, the statement
      (1) is proved. 

  (2) It is straightforward to check the relations
      (\ref{eq:33}). Using these, any element $c$ in the division
      algebra $B(\F_{p^n})[\Pi']$ can be written uniquely as
\[ c=\Pi^r (c_0+c_1\Pi+\dots+ c_{n-1} \Pi^{n-1}),  \]
for some $r\in \Z$ and some elements $c_i\in W(\F_{p^n})$ for
$i=0,\dots, n-1$ 
such that  $c_0$ is a unit in $W(\F_{p^n})$. The valuation $v_p(c)$ is 
$r/n$. This shows
that the subring $W(\F_{p^n})[\Pi]$ consists of elements $c$ with
$v_p(c)\ge 0$. Since any order of $B(\F_{p^n})[\Pi']$ is contained in the
subring of elements $c$ with $v_p(c)\ge 0$, the order
$W(\F_{p^n})[\Pi]$ is maximal. \qed 
\end{proof}

According to Lemma~\ref{31}, a \dieu module $M$ or an $F$-isocrystal is
isoclinic if and only if its endomorphism algebra is a
(finite-dimensional) central
simple algebra over $\Qp$. A \dieu module $M$ or $M\otimes \Qp$ is called 
{\it isosimple} if its endomorphism algebra is a (finite-dimensional)
central division algebra over $\Q_p$, that is, 
the $F$-isocrystal $M\otimes \Q_p$ is isomorphic to 
$M_{(a,b)}\otimes \Qp$ for some pair $(a,b)$.   

\subsection{Minimal \dieu modules}
\label{sec:32}
Let $(a,b)$ be a pair as above, and let $n:=a+b$. Denote by
$\bfM_{(a,b)}$ the \dieu module over $\Fp$ as follows: it is generated by
elements $e_i$, for $i\ge 0\in \Z$, with relation $e_{i+n}=pe_i$, 
as a $\Zp$-module, 
and with operations $Fe_i=e_{i+b}$ and $Ve_i=e_{i+a}$ for all $i\in
\Z_{\ge 0}$. One extends the maps $F$ and $V$ on $\bfM_{(a,b)}\otimes
W$ by $\sigma$-linearity and $\sigma^{-1}$-linearity, respectively, so
that $\bfM_{(a,b)}\otimes W$ is a \dieu module over $k$. 

Let $\beta=\sum_{i} r_i (a_i,b_i)$ be a Newton polygon. We put
$\bfM(\beta):=\sum_{i} \bfM_{(a_i, b_i)}^{\oplus {r_i}}$. Note that
the \dieu module $\bfM(\beta)$ has Newton polygon $\beta$. Write
$\beta^t:=\sum_{i} r_i (b_i,a_i)$ for the dual of $\beta$.
Denote by
$\bfH(\beta)$ the $p$-divisible group over $\Fp$ corresponding to the
\dieu module $\bfM(\beta^t)$ (This is because we use the covariant
theory; the Newton polygon of a $p$-divisible group $G$ is equal to
the dual of that of its \dieu module $M(G)$).  

\begin{defn} (\cite{dejong-oort:purity}, Section 5) \

 (1) A \dieu module $M$ over $k$ is called {\it minimal} if it is
      isomorphic to $\bfM(\beta)\otimes W$ for some Newton polygon
      $\beta$. In this case
      $\beta$ is the Newton polygon of $M$.


 (2) A $p$-divisible group $X$ over $k$ is called {\it minimal} if its
     associated \dieu module is so.
\end{defn}

Let $M_\lambda$ be an isoclinic \dieu module of slope
$\lambda=\frac{b}{a+b}$ (in reduced form). There exist integers  $x$
and $y$ such that $xa+yb=1$. Put $N_\lambda:=M_\lambda\otimes \Q_p$ and let
$\Pi_0:=F^y V^x$ be an operator on $N_\lambda$; it is
$\sigma^{y-x}$-linear and it depends on the choice of the integers $x$
and $y$. Let 
\begin{equation}
  \label{eq:34}
  \wt N_\lambda:=\{m\in N_\lambda\, |\, F^n m=p^b m \, \}
\end{equation}
be the skeleton of $N_\lambda$; it is a $B(\F_{p^n})$-subspace that has
the same dimension as $N_\lambda$, equivalently $\wt N_\lambda$
generates $N_\lambda$ over $B(k)$. Since $\Pi_0 F=F\Pi_0$, the
operator $\Pi_0$ leaves the subspace $\wt N_\lambda$ invariant. The
restriction of $\Pi_0$ to $\wt N_\lambda$  has the 
following properties:

\begin{itemize}
\item $\Pi_0$ (on $\wt N_\lambda$) is independent of the choice of the
  integers $x$ and $y$, and  
\item $\Pi_0^n=p$, $\Pi_0^b=F$ and $\Pi_0^a=V$ on $\wt N_\lambda$.
\end{itemize}


\begin{lemma}\label{33}
  Notation as above. An isoclinic \dieu module $M_\lambda$ of slope
  $\lambda$ is minimal
  if and only if (i) $F^n M_\lambda=p^b M_\lambda$, and (ii)
  $\Pi_0(M_\lambda)\subset M_\lambda$.
\end{lemma}
\begin{proof}
  It is clear that a minimal isoclinic \dieu module satisfies
  the conditions (i) and (ii). Conversely, suppose $M_\lambda$ satisfies
  the conditions (i) and (ii). The condition (i) implies that
  $M_\lambda$ is generated by the skeleton $\wt M_\lambda$ over
  $W$. Since $\Pi_0 \wt M_\lambda\supset \Pi_0^n \wt M_\lambda=p  \wt
  M_\lambda$, the quotient $\wt M_\lambda/\Pi_0(\wt M_\lambda)$ is a
  finite-dimensional vector space over $\F_{p^n}$.
  Choose elements $f_1,\dots, f_d$ in $\wt M_\lambda$ such 
  that they form an $\F_{p^n}$-basis in 
  $\wt M_\lambda/\Pi_0(\wt M_\lambda)$. For each $i=1,\dots,d$, the
  $W$-submodule $\<f_i, \Pi_0(f_i), \dots, \Pi_0^{n-1}(f_i)\>$ is a
  \dieu submodule of $M$. Since 
  $F$ sends $\Pi_0^j(f_i)$ to $\Pi_0^{j+b}(f_i)$ and $V$ sends
  $\Pi_0^j(f_i)$ to $\Pi_0^{j+a}(f_i)$, this \dieu module is 
  isomorphic to $W\otimes \bfM_{(a,b)}$ by sending $\Pi_0^j(f_i)$ to
  $e_j$. Therefore, $M_\lambda\simeq W\otimes \bfM_{(a,b)}^{\oplus
  d}$. This proves the lemma. \qed   
\end{proof}
  
Let $M$ be a \dieu module. Put $N:=M\otimes \Q_p$. Let 
\[ N=\bigoplus_{\lambda} N_\lambda \]
be the decomposition into isotypic components. Put $M_\lambda:=M\cap
N_\lambda$. 


\begin{lemma}\label{34}\

  (1) A \dieu module $M$ is minimal if and only if its endomorphism
      ring $\End(M)$ is a maximal order of $\End^0(M)$.

  (2) A \dieu module $M$ is minimal if and only if it is isomorphic to
      the direct sum of its isotypic components $M_\lambda$ and each
      factor $M_\lambda$ is minimal.
\end{lemma}
\begin{proof}
  (1) To prove the only if part, it suffices to show when
      $M=\bfM_{(a,b)}$ (for simplicity we write $\bfM_{(a,b)}$
      for $\bfM_{(a,b)}\otimes W$ here). 
      Let $n:=a+b$ and $m\in \Z$ such that $mb\equiv
      1 \mod n$. For each element $c\in W(\F_{p^n})$, we define an
      endomorphism $\varphi_c\in \End_{\calD \calM}(\bfM_{(a,b)})$ by
      $\varphi_c(e_i)=\sigma^{mi}(c) e_i$ for all $i\ge 0$. Let
      $\Pi\in \End_{\calD \calM}(\bfM_{(a,b)})$ be the endomorphism which
      sends $e_i$ to $e_{i+1}$. The endomorphism ring $\End_{\calD
      \calM}(\bfM_{(a,b)})$ is generated by elements $\Pi$ and
      $\varphi_c$ for all $c\in  W(\F_{p^n})$, subject to the
      relations $\Pi^n=p$ and $\Pi \varphi_c= \varphi_{\sigma^{-m}(c)}
      \Pi$. Hence, 
      $\End_{\calD \calM}(\bfM_{(a,b)})\simeq  W(\F_{p^n})[\Pi]$ with
      relations $\Pi^n=p$ and $\Pi c \Pi^{-1}=\sigma^{-m}(c)$ for
      $c\in W(\F_{p^n})$. This is the maximal order in the
      endomorphism algebra $\End^0_{\calD \calM}(\bfM_{(a,b)})$; see
      Lemma~\ref{31}. 

      We prove the if part. First of all, a maximal order is isomorphic to a
      product of matrix rings $M_d(O_D)$, where $D$ is a division
      central algebra over $\Qp$ and $O_D$ is its maximal order.
      Using the Morita equivalence, we can assume that
      $\End^0_{\calD\calM}(M)$ is a division algebra $D$ and
      $\End_{\calD\calM}(M)=O_D$. Let
      $[D:\Qp]=n^2$. One chooses a presentation for
      $O_D=W(\F_{p^n})[\Pi]$ with 
      relations $\Pi^n=p$ and $\Pi c \Pi^{-1}=\sigma^{-m}(c)$ for
      $a\in W(\F_{p^n})$, for some $m\in \Z$. Let $b$ be the integer
      such that $bm\equiv 1\ \mod n$ 
      and $0\le b <n$. Using Lemma~\ref{31}, the division algebra $D$
      has invariant $b/n$, and hence the \dieu module $M$
      has single slope $b/n$. Put $\wt M:=\{x\in M; F^nx=p^b
      x\}$ and $\wt N:=\wt {M\otimes \Q_p}$ be the skeleton of
      $M\otimes \Qp$. It follows from $F\Pi=\Pi
      F$ that $\Pi$ is an automorphism on $\wt N$. It follows
      from $\wt N\cap M=\wt M$ that for $x\in M$, one has $x\in \wt M$ 
      if and only if $\Pi x\in
      \wt M$; this implies $\wt M\not\subset \Pi M$. Choose an element
      $e_0\in \wt M \setminus \Pi M$. Then elements 
      $e_0, \Pi(e_0),\dots, \Pi^{n-1}(e_0)$ generate $M$ over
      $W$. Using 
      $F \Pi=\Pi F$ and $F^n=p^b$ on $\wt M$, one can show that
      $F(e_0)=\alpha \Pi^{b} (e_0)$ for some $\alpha\in
      W(\F_{p^n})^\times$ with $N_{W(\F_{p^n})/\Zp}(\alpha)=1$. By
      Hilbert's 90, one may replace $e_0$ by $\lambda e_0$ so that
      $F(e_0)=\Pi^{b} (e_0)$. This shows $M\simeq \bfM_{(a,b)}$. 

     (2) This is clear. \qed 
\end{proof}

\section{Construction of minimal isogenies}
\label{sec:04}

\subsection{Minimal isogenies}
\label{sec:41}

\begin{defn}\label{41} (cf. \cite{li-oort}, Section 1)
  Let $X$ be a $p$-divisible group over $k$. The minimal isogeny of
  $X$ is a pair $(X_0, \varphi)$ where $X_0$ is a minimal $p$-divisible
  group over $k$, and $\varphi:X_0\to X$ is an isogeny over $k$ such
  that for any other pair $(X_0',\varphi')$ as above there exists an
  isogeny $\rho: X'_0\to X_0$ such that
  $\varphi'=\varphi\circ\rho$. Note that the morphism $\rho$ is unique
  if it exists.
\end{defn}

\begin{lemma}\label{42}
  Let $M$ be a \dieu module over $k$. Then there exists a unique
  biggest minimal \dieu submodule $M_{\rm min}$ contained in
  $M$. Dually there is a unique smallest minimal \dieu module $M^{\rm
  min}$ containing $M$. 
\end{lemma}
\begin{proof}
  Suppose that $M_1$ is a minimal \dieu module contained in $M$. Then
  $M_{1,\lambda}\subset M_\lambda$ (see Section 3). Therefore we may assume
  that $M$ is isoclinic of slope $\lambda$. If $M_1$ and $M_2$ are
  two minimal \dieu modules contained $M$, then $M_1+M_2$ satisfies
  the conditions (i) and (ii) in Lemma~\ref{33}, and hence it is
  minimal. This completes the proof. \qed
\end{proof}

The minimal \dieu module $M_{\rm min}$ is called the {\it minimal \dieu
submodule} of $M$; the module $M^{\rm min}$ is called the {\it minimal
\dieu overmodule} of $M$. By Lemma~\ref{42}, we have 

\begin{cor}\label{43}
  For any $p$-divisible group $X$ over $k$, the minimal isogeny exists. 
\end{cor}

\begin{remark}\label{44}
For the reader who might question about ground fields, we mention that
the notion of minimal isogenies can be generalized over any
field of \ch $p$ as follows. Let $X$ be a $p$-divisible group over a
field $K$ of \ch $p>0$. We call a $K$-isogeny $\varphi:X_0\to X$ {\it
  minimal} if
\begin{itemize}
\item [(i)] (stronger form) $X_0$ is isomorphic to $\bfH(\beta)\otimes_{\Fp}
  K$, for some Newton polygon $\beta$, and $\varphi$ satisfies the universal
  property as in Definition~\ref{41}, or
\item [(ii)] (weaker form) the base change over its algebraic closure
  $\varphi_{\bar K}:X_{0,\bar K} \to X_{\bar K}$ is the minimal
  isogeny of $X_{\bar K}$.  
\end{itemize}

Suppose that $X$ is an etale $p$-divisible group over $K$. Then
$X_{\bar K}$ is a minimal $p$-divisible group, and the identity map
$id:X\to X$ is a minimal isogeny in the sense of the weaker form. 
However, if $X$ is not isomorphic to the constant etale $p$-divisible
group, then $X$ is not isogenous over $K$ to the constant etale
$p$-divisible group. Therefore, $X$ does not admit a minimal isogeny
in the sense of the stronger form. 
\end{remark}

We need the following finiteness result due to Manin. This
follows immediately from \cite{manin:thesis}, Theorems III.3.4 and
III.3.5. 

\begin{thm}\label{45}
  Let $h\ge 1$ be a positive integer. Then there is an integer $N$
  only depending on $h$ such that for all $p$-divisible groups $X$ 
  of height $h$ over $k$, the degree of the minimal isogeny $\varphi$ 
  of $X$ is less than $p^N$. 
\end{thm}

\begin{remark}
  Let $E$ be the (unique up to isomorphism) supersingular
  $p$-divisible group of height two over $k$, and let $X_0:=E^g$. Let
  $X$ be a supersingular $p$-divisible group of height $2g$ over
  $k$. Nicole and Vasiu showed that the kernel of the minimal isogeny
  $\varphi: X_0\to X$ is annihilated by $p^{\lceil (g-1)/2\rceil}$;
  see \cite{nicole-vasiu:indiana2007}, Remark~2.6 and
  Corollary~3.2. Moreover, this is optimal, that is, there is a
  supersingular $p$-divisible group $X$ of height $2g$ such that $\ker
  \varphi$ is not annihilated by $p^{\lceil (g-1)/2\rceil-1}$; see
  \cite{nicole-vasiu:indiana2007}, Example~3.3.   
\end{remark}
\subsection{Construction of minimal isogenies}
\label{sec:42}
Let $M$ be a \dieu module over $k$. Put $N:=M\otimes \Q_p$ and let 
\[ N=\bigoplus_\lambda N_\lambda \]
be the isotypic decomposition. Let $\wt N_\lambda$ be the skeleton of
$N_\lambda$ (see (\ref{eq:34})) and put $\wt M_\lambda:=M_\lambda\cap
\wt N_\lambda$. Let $(a,b)$ be the pair associated to $\lambda$ and
put $n=a+b$. Write $W_0$ for the ring $W(\F_{p^n})$ of Witt vectors
over $\F_{p^n}$. 
Let $\wt Q_\lambda:=W_0[\Pi_0] \wt M_\lambda^t$, the
$W_0[\Pi_0]$-submodule of $\wt N^t_\lambda$ generated by $\wt
M_\lambda^t$. Let $\wt P_\lambda:=\wt Q_\lambda^t$ and let
\begin{equation}
  \label{eq:41}
  P(M):=\bigoplus_\lambda \<\wt P_\lambda\>_W. 
\end{equation}
We claim that
\begin{lemma}\label{46}
  The \dieu module $P(M)$ constructed as above is the minimal \dieu
  submodule $M_{\rm min}$ of $M$.
\end{lemma}
\begin{proof}
  It is clear that $M_{\rm min}=\oplus M_{\rm min,\lambda}$ and
  $M_{\rm min,\lambda}$ is the minimal \dieu submodule of
  $M_\lambda$. Therefore, it suffices to check 
  $\wt P_\lambda=\wt M_{\rm min,\lambda}$. As $\wt M_{\rm
  min,\lambda}\subset \wt M_\lambda$, $\wt M_{\rm min,\lambda}$ is the
  minimal \dieu submodule of $\wt M_\lambda$. Taking dual, it suffices
  to show that $\wt Q_\lambda$ is the minimal \dieu overmodule of $\wt
  M_\lambda^t$. This then follows from Lemma~\ref{33}. \qed  
\end{proof}

Let $\calO$ be an order of a finite-dimensional semi-simple algebra
over $\Q_p$. A $p$-divisible $\calO$-module is a pair $(X,\iota)$,
where $X$ is a $p$-divisible group and $\iota:\calO\to \End(X)$ is a
ring monomorphism. 

\begin{prop}\label{47}
  Let $(X,\iota)$ be a $p$-divisible $\calO$-module over $k$ and let
  $\varphi:X_0\to X$ be the minimal isogeny of $X$ over $k$. Then
  there is a unique ring monomorphism $\iota_0:\calO\to \End(X_0)$
  such that $\varphi$ is $\calO$-linear. 
\end{prop}
\begin{proof}
  Let $M$ be the \dieu module of $X$ and let 
  $\phi\in \End_{\DM}(M)$ be an endomorphism. It suffices to show that
  $\phi(M_{\rm min})\subset M_{\rm min}$. It is clear that $\phi(\wt
  N_\lambda)\subset \wt N_\lambda$. It follows from the construction
  of the minimal \dieu submodule that $\phi(M_{\rm min})\subset M_{\rm
  min}$. This proves the proposition. \qed
\end{proof}

\subsection{Proof of Theorem~\ref{25}}
\label{sec:43}

Let $M$ be the \dieu module of $X$. Let $M^{\rm min}$ be the minimal
\dieu overmodule of $M$. By Theorem~\ref{45}, there is a positive
integer $N_1$ only depending on the rank of $M$ such that the length
${\rm length}(M^{\rm min}/M)$ as a $W$-module is less than $N_1$. Let
$N_2$ be a positive integer 
so that $p^{N_2}M^{\rm min}\subset M \subset M^{\rm min}$. Let 
$\phi\in \End_{\DM}(M)$ be an element.  
By Proposition~\ref{47}, one has
$\phi\in  \End_{\DM}(M^{\rm min})$. Therefore, we have showed
\begin{equation}
  \label{eq:42}
  \End_{\DM}(M)=\{\phi\in  \End_{\DM}(M^{\rm min})\, ; \,
\phi(M)\subset M  \}.
\end{equation}
We claim that $p^{N_2}\End_{\DM}(M^{\rm min})\subset
\End_{\DM}(M)$. Indeed, if $\phi\in \End_{\DM}(M^{\rm min})$, then 
\[ p^{N_2}\phi(M)\subset p^{N_2} M^{\rm min}\subset M. \]
Therefore, there is an positive integer $N$ only depending on the rank
of $M$ such that $v_p(\ci(\End_{\DM}(M)))< N$. This completes the proof
of Theorem~\ref{25}, and hence completes the proof of Theorem~\ref{11}.

\subsection{Proof of Theorem~\ref{12}}
\label{sec:44}
By a theorem of Tate \cite{tate:eav}, we have 
\[ [\End^0(A)\otimes \Q_p:\Q_p]\le 4 g^2. \]
Since there are finitely many finite extensions of $\Q_p$
of bounded degree, and finitely many Brauer invariants with bounded 
denominator, there are finitely many semi-simple algebras
$\End^0(A)\otimes \Q_p$, up to isomorphism, 
of abelian varieties $A$ of dimension $g$. It follows from
Theorem~\ref{11} that in each isogeny class there are finitely many
endomorphism rings $\End(A)\otimes \Z_p$, up to
isomorphism. Therefore, there are finitely many isomorphism classes of 
the endomorphism rings $\End(A)\otimes \Z_p$ for all
$g$-dimensional abelian varieties $A$ of a field of \ch $p>0$. This
completes the proof. 











\section{Examples}
\label{sec:05}

\subsection{}
\label{sec:51}
We start with a trivial example. Suppose the abelian variety $A_0$
over a field $k$ has the property $\End_k(A_0)=\Z$. Then for any
member $A\in [A_0]_k$, the endomorphism ring $\End_k(A)$ is always
a maximal order. Therefore, there is an isogeny class $[A_0]_k$ such
that the endomorphism rings $\End_k(A)$ are maximal for all $A\in
[A_0]_k$.

\subsection{}
\label{sec:52}
Let $p$ be any prime number. 
Let $K$ be an imaginary quadratic field such that $p$ splits in $K$. 
Let $O_K$ be the ring of integers. For any positive integer $m$, let 
$E^{(m)}$ be the elliptic curve over $\C$ so that $E^{(m)}(\C)=\C/\Z+m
O_K$. It is easy to see that $\End_{\C} (E^{(m)})=\Z+mO_K$, and hence
$\ci(\End(E^{(m)}))=m$. By the 
theory of complex multiplication 
\cite{shimura:aaf1971}, each elliptic curve $E^{(m)}$ is defined 
over $\bar \Q$ and has good reduction everywhere over some number
field. Let $E^{(m)}_p$ be the reduction of $E^{(m)}$ over $\Fpbar$; this
is well-defined. Since $O_K\otimes \Zp$ has non-trivial idempotent and
hence it is not contained in the division quaternion $\Q_p$-algebra, 
$E^{(m)}_p$ is ordinary. Therefore, we have 
$\End(E^{(m)}_p)\otimes \Zp=\Z_p\times \Z_p$ is maximal (see
\cite{deuring}, cf.  
\cite{lang:ef}, Chapter 13, Theorem 5, p.~175). 
Clearly we have $E^{(m)}\in [E^{(1)}]_{\Qbar}$ and 
$E^{(m)}_p\in [E^{(1)}_p]_{\Fpbar}$. Using \cite{oort:endo}, Lemma
2.1, we show that for $(m,p)=1$,
$\ci(\End(E^{(m)}_p))=\ci(\End(E^{(m)})=m$. 
These give examples over a field \ch zero or $p>0$ in Proposition~\ref{14}.  

Note that not all of elliptic curves $E^{(m)}$ (resp. $E^{(m)}_p$)
above are defined over a fixed  number field (resp. a fixed finite
field). Therefore, we did not exhibit an example for
Proposition~\ref{14} when the ground $k$ is of finite type over 
its prime field. 
It is natural to ask: if
$k$ is finitely generated over its prime field, are $\ell$-co-indices 
$v_\ell(\ci(\End(A)))$, for all $A\in [A_0]_k$, bounded or unbounded?

\subsection{Description of $\calS_2$}
\label{sec:53}
Let $n\ge 3$ be a prime-to-$p$ positive integer. Let
$\calA_{2,1,n}\otimes \Fpbar$ denote the Siegel 3-fold over $\Fpbar$
with level 
$n$-structure, and let $\calS_2$ denote the supersingular locus. Let
$\Lambda^*$ be the set of isomorphism classes of superspecial
polarized abelian surfaces $(A,\lambda,\eta)$ over $\Fpbar$ 
with polarization degree $\deg \lambda=p^2$ and a level 
$n$-structure $\eta$. 
For each member $\xi=(A_1,\lambda_1,\eta_1)\in \Lambda^*$, let 
$S_\xi$ be the space 
that parametrizes degree $p$ isogenies
$\varphi:(A_1,\lambda_1,\eta_1)\to (A,\lambda,\eta)$ preserving
polarizations and level structures. The variety $S_\xi$ is isomorphic
to $\bfP^1$ over $\Fpbar$; we impose the $\F_{p^2}$-structure on
$\bfP^1$ defined by $F^2=-p$ on $M_1$, where $M_1$ is the \dieu module
of $A_1$ and $F$ is the Frobenius map on $M_1$. For this structure,
the superspecial points are exactly the $\F_{p^2}$-valued points on $V$. 
It is known (see Katsura-Oort \cite{katsura-oort:surface}) 
that the projection ${\rm pr}: S_\xi \to \calS_2$ induces an
isomorphism ${\rm pr}:S_\xi\simeq V_\xi\subset \calS_2$ 
onto one irreducible component. Conversely, any
irreducible component $V$ is of the form $V_\xi$ for exact one
member $\xi \in \Lambda^*$. Two irreducible components $V_1$ and
$V_2$, if they intersect, intersect transversally at some superspecial
points.  

\subsection{``Arithmetic'' refinement of $\calS_2$}
\label{sec:54}
We describe the arithmetic refinement on one irreducible component
$V=\bfP^1$ of $\calS_2$. For any point $x$, we write $c_p(x)$ for
$v_p(\ci(\End(A_x)))$, where $A_x$ is the underlying abelian surface
of the object $(A_x,\lambda_x,\eta_x)$ corresponding to the point
$x$. Let $D$ be the division quaternion
algebra over $\Q_p$ and let $O_D$ be the maximal order of $D$. 
The endomorphism ring of a superspecial \dieu module is (isomorphic
to) $M_2(O_D)$. For
non-superspecial supersingular \dieu modules, one can compute their
endomorphism rings using (\ref{eq:42}). 
Let $\pi:M_2(O_D)\to M_2(\F_{p^2})$ be the
natural projection. We compute these endomorphism rings and get (see
\cite{yu-yu:mass_surface}, Proposition 3.2):
  
\begin{prop}\label{51}
  Let $x$ be a point in $V=\bfP^1$ and let $M_x$ be the associated
  \dieu module. 

  {\rm (1)} If $x\in \bfP^1(\F_{p^2})$, then $\End_{\DM}(M_x)=M_2(O_D)$.

  {\rm (2)} If $x\in \bfP^1(\F_{p^4}) -\bfP^1(\F_{p^2})$, then 
\begin{equation*}
  \End_{\DM}(M_x)\simeq \{\phi\in M_2(O_D)\, ; \pi(\phi)\in B'_0\, \},
\end{equation*}
where $B_0'\subset M_2(\F_{p^2})$ is a subalgebra
isomorphic to $\F_{p^2}(x)$.

  {\rm (3)} If $x\in \bfP^1(k)-\bfP^1(\F_{p^4})$, then 
\begin{equation*}
  \End_{\DM}(M_x)\simeq \left \{\phi\in M_2(O_D)\, ; 
  \pi(\phi)=  \begin{pmatrix}
   a & 0 \\ 0 & a
  \end{pmatrix},\, a\in \F_{p^2} \, \right \}.
\end{equation*}
\end{prop}

We remark that Proposition~\ref{51} was also known to Ibukiyama. 

For each integer $m\ge 0$, let
\[ V_m:=\{x\in V; c_p(x)\le m\}. \]
The collection $\{V_m\}_{m\ge 0}$ forms an increasing sequence of 
closed subsets of $V=\bfP^1$. We apply Proposition~\ref{51} and get
\[ V_0=\dots=V_3\subset V_4=V_5\subset V_6=V, \]
and 
\[ V_0=\bfP^1(\F_{p^2}), \quad  V_4=\bfP^1(\F_{p^4}). \]
This provides more information on $\calS_2$ not just superspecial and
non-superspecial points.




\subsection{Semi-simplicity of Tate modules}\label{sec:55}
Let $A$ be an abelian variety over a field $k$. Let $k^{sep}$ a
separable closure of $k$ and let $G:=\Gal(k^{sep}/k)$ be the Galois
group. To each prime  $\ell\neq \char (k)$, one associates
the $\ell$-adic Galois representation 
\[ \rho_\ell:G\to \Aut (T_\ell(A)), \] 
where $T_\ell(A)$ is the Tate module of $A$. 
According to Faltings \cite{faltings:end} and Zarhin
\cite{zarhin:end}, under the condition that 
the ground field $k$ is of finite type over its
prime field, the Tate  module $V_\ell:=T_\ell(A)\otimes \Q_\ell$ 
is semi-simple as a $\Q_\ell[G]$-module. 
We show that this condition is necessary.

Let $A_0$ be an abelian variety over a field $k_0$ which is finitely
generated over its prime field. We write $G_0:=\Gal(k_0^{\rm
  sep}/k_0)$ and $G^{\rm alg}_0$ for the algebraic envelope
of $G_{\ell}:=\rho_\ell(G_0)$; that is, $G^{\rm alg}_0$ is the Zariski
closure of $G_\ell$ in $\Aut(V_\ell(A_0))$ that is regarded as
algebraic groups over $\Q_\ell$. 
Assume that the algebraic group $G^{\rm alg}_0$
is not a torus; for example let $A_0$ be an elliptic curve without CM. 
We shall choose an intermediate
subfield $k_0\subset k\subset k_0^{\rm sep}$ so that the Tate module
$V_\ell(A)$ associated to the base change $A:=A_0\otimes k$ is not
semi-simple as a $G:=\Gal(k^{\rm sep}/k)$-module. 
We can choose a closed subgroup $H\subset G_{\ell}$
such that $V_\ell(A_0)$ 
is not a semi-simple $\Q_\ell[H]$-module. To see this, 
by Bogomolov's theorem (see \cite{bogomolov:alg}),  $G_{\ell}$ is an
open compact subgroup of $G^{\rm alg}_0(\Q_\ell)$. 
We choose a Borel subgroup of $B$ of $G^{\rm
  alg}_0$ and let $H$ be the the intersection $G_\ell \cap
B(\Q_\ell)$. Then $H$ is a closed non-commutative solvable group and
$V_\ell(A_0)$ is not a semi-simple $\Q_\ell[H]$-module. 
Using the Galois theory, let $k$ correspond the closed subgroup 
$\rho_\ell^{-1}(H)$. Then the abelian variety $A:=A_0\otimes k$ gives
a desired example. 

In this example, the endomorphism algebra
$\End^0_k(A[\ell^\infty])=\End_{\Q_\ell [H]}(V_\ell(A))$ is not
semi-simple.

\subsection{Semi-simplicity of endomorphism algebras of $p$-divisible
  groups}
\label{sec:56}
Let $k$ be a field of \ch $p>0$. Consider the following two questions:\\

(1) Is the category of $p$-divisible groups of finite height 
    up to isogeny over k semi-simple?\\

(2) Is the endomorphism algebra $\End_k(X)\otimes \Qp$ of a
    $p$-divisible group $X$ over $k$ semi-simple? \\

We show that the answers to both questions are negative. Indeed etale
$p$-divisible groups already provide such examples. Note that the
category of etale $p$-divisible groups of finite height up to isogeny
is equivalent to the category of continuous linear representations of
$\Gal(k^{\rm sep}/k)$ on finite-dimensional $\Qp$-vector spaces. 
For instance, one can have a 2-dimensional Galois representation whose
image is the set of all upper-triangular unipotent matrices in 
$\GL_2(\Zp)$. It gives a counter-example for both questions. \\ \\

\npr {\bf Erratum. } 
The second assertion of Lemma~\ref{22} is wrong.   
As this is only used to reduce to \ac fields, the statements of
Theorems~\ref{11} and \ref{12} are valid at least when the ground field
is \ac of \ch $p>0$. 

To see this false, for example, take a supersingular elliptic curve
$E$ over $\F_7$ whose endomorphism ring is $\Z[\sqrt{7}]$, which is not a
maximal order. However, the endomorphism ring of
$E\otimes \F_{p^2}$ is a maximal order of the quaternion algebra over
$\Q$ ramified at $\{7,\infty\}$.


\begin{thebibliography}{99}
\def\jams{{\it J. Amer. Math. Soc.}} 
\def\invent{{\it Invent. Math.}} 
\def\ann{{\it Ann. Math.}} 
\def\ihes{{\it Inst. Hautes \'Etudes Sci. Publ. Math.}} 

\def\ecole{{\it Ann. Sci. \'Ecole Norm. Sup.}}
\def\ecole4{{\it Ann. Sci. \'Ecole Norm. Sup. (4)}} 
\def\mathann{{\it Math. Ann.}} 
\def\duke{{\it Duke Math. J.}} 
\def\jag{{\it J. Algebraic Geom.}} 
\def\advmath{{\it Adv. Math.}}
\def\compos{{\it Compositio Math.}} 
\def\ajm{{\it Amer. J. Math.}} 
\def\grenoble{{\it Ann. Inst. Fourier (Grenoble)}}
\def\crelle{{\it J. Reine Angew. Math.}}
\def\mrl{{\it Math. Res. Lett.}}
\def\imrn{{\it Int. Math. Res. Not.}}
\def\acad{{\it Proc. Nat. Acad. Sci. USA}}
\def\tams{{\it Trans. Amer. Math. Sci.}}
\def\cras{{\it C. R. Acad. Sci. Paris S\'er. I Math.}} 
\def\mathz{{\it Math. Z.}} 
\def\cmh{{\it Comment. Math. Helv.}}
\def\docmath{{\it Doc. Math. }}
\def\asian{{\it Asian J. Math.}}
\def\jussieu{{\it J. Inst. Math. Jussieu}} 

\def\manmath{{\it Manuscripta Math.}} 
\def\jnt{{\it J. Number Theory}} 
\def\ijm{{\it Israel J. Math.}}
\def\ja{{\it J. Algebra}} 
\def\pams{{\it Proc. Amer. Math. Sci.}}
\def\smfmemoir{{\it Bull. Soc. Math. France, Memoire}}
\def\bsmf{{\it Bull. Soc. Math. France}}
\def\sb{{\it S\'em. Bourbaki Exp.}}
\def\jpaa{{\it J. Pure Appl. Algebra}}
\def\jems{{\it J. Eur. Math. Soc. (JEMS)}}
\def\jtokyo{{\it J. Fac. Sci. Univ. Tokyo}}
\def\cjm{{\it Canad. J. Math.}}
\def\jaums{{\it J. Australian Math. Soc.}}
\def\pspm{{\it Proc. Symp. Pure. Math.}}
\def\ast{{\it Ast\'eriques}}
\def\pamq{{\it Pure Appl. Math. Q.}}
\def\nagoya{{\it Nagoya Math. J.}}
\def\forum{{\it Forum Math. }}

\def\tp{{To appear}}

\newcommand{\princeton}[1]{Ann. Math. Studies #1, Princeton
  Univ. Press}

\newcommand{\LNM}[1]{Lecture Notes in Math., vol. #1, Springer-Verlag}

\bibitem{bogomolov:alg} 
F.A.~Bogomolov, Sur l'alg\'ebricit\'e des repr\'esentations
$l$-adiques. \cras~{\bf 290} (1980), no. 15, A701--A703.

\bibitem{dejong:homo}  A.J. de Jong, Homomorphisms of Barsotti-Tate
  groups and crystals in positive characteristic. \invent~{\bf 134}
  (1998), 301--333. 

\bibitem{dejong-oort:purity} A.J. de Jong and F. Oort, Purity of the
  stratification by Newton polygons. \jams~{\bf 13} (2000), 209--241.

\bibitem{deuring} M. Deuring, Die Typen der Multiplikatorenringe
  elliptischer Funktionenk\"orper. 
  {\it Abh. Math. Sem. Hamburg}~{\bf 14} (1941), 197--272.

\bibitem{faltings:end} G.~Faltings, Endlichkeitss\"atze f\"ur abelsche
  Variet\"aten \"uber Zahlk\"orpern. \invent~{\bf 73} (1983), 349--366. 

\bibitem{katsura-oort:surface} T. Katsura and F. Oort, Families of
  supersingular abelian surfaces, \compos~{\bf 62} (1987), 107--167.

\bibitem{lang:ef} S. Lang, {\it Elliptic functions.} 
 Second edition.  Graduate Texts in Mathematics, {\bf 112}. 
 {\it Springer-Verlag, New York}, 1987.  

\bibitem{li-oort} K.-Z. Li and  F. Oort, {\it Moduli of Supersingular 
  Abelian Varieties.} \LNM{1680}, 1998.



\bibitem{manin:thesis} Yu. Manin, Theory of commutative formal groups
  over fields of finite characteristic. 
  {\it Russian Math. Surveys}~{\bf 18} (1963), 1--80.


\bibitem{oort:cm} F. Oort, The isogeny class of a CM-type abelian
  variety is defined over a finite extension of the prime
  field. \jpaa~{\bf 3} (1973), 399--408.

\bibitem{oort:endo} F. Oort, Endomorphism algebras of abelian
  varieties. {\it Algebraic geometry and commutative algebra, in honor
  of M. Nagata} (1988), 469--502.

\bibitem{oort:minimal} F.~Oort, Minimal $p$-divisible groups. 
 \ann~{\bf 161} (2005), 1021--1036. 

\bibitem{nicole-vasiu:indiana2007} M.-H.~Nicole and A.~Vasiu,
  Minimal truncations of supersingular $p$-divisible groups. {\it Indiana
  Univ. Math. J.}~{\bf 56} (2007), 2887--2897. 

\bibitem{pierce} R.~S. Pierce, {\it Associative algebras}. 
   Graduate Texts in Mathematics,~{\bf 88}. Springer-Verlag, New
   York-Berlin, 1982. 436 pp. 

\bibitem{reiner:mo} I.~Reiner, {\it Maximal orders}. 
  London Mathematical Society Monographs, No.~{\bf 5}.~Academic Press,
  London-New York, 1975. 395 pp.  

\bibitem{rotger:mrl2008} V.~Rotger, Which quaternion algebras act on a
  modular abelian variety? \mrl~{\bf 15} (2008), no. 2,
  251--263. 

\bibitem{serre:1984-5} J.-P.~Serre, R\'esum\'e des cours de 1984-1985.
   pp.~27--32, \OE uvres Collected papers. IV.~1985--1998, 2000.  
 
\bibitem{shimura:aaf1971} G. Shimura, {\it Introduction to the
  Arithmetic Theory of Automorphic Functions.} Publ. Math. Soc. Japan
  11 (Iwanami Shoten, Tokyo, and Princeton Univ. Press, 1971).

\bibitem{tate:eav} J. Tate, Endomorphisms of abelian varieties over
  finite fields. \invent~{\bf 2} (1966), 134--144.


\bibitem{waterhouse:thesis} W.~C.~Waterhouse, Abelian
  varieties over finite fields. \ecole4~{\bf 1969}, 521--560.  

\bibitem{yu:cm} C.-F. Yu, The isomorphism classes of abelian varieties of
  CM-type. \jpaa~{\bf 187} (2004) 305--319.


\bibitem{yu-yu:mass_surface} C.-F. Yu and J.-D. Yu, Mass formula for
  supersingular abelian surfaces. \ja~{\bf 322} (2009), 3733--3743.


\bibitem{zarhin:end} J.G.~Zarhin, Endomorphisms of Abelian varieties over
  fields of finite characteristic. {\it Izv. Akad. Nauk SSSR
  Ser. Mat.}~{\bf 39} (1975), 272--277 = {\it Math. USSR-Izv.}~{\bf 9}
  (1975), 255--260.   


\end{thebibliography}
\end{document}